\documentclass[12pt,a4paper]{amsart}
\usepackage{amsfonts}
\usepackage{amssymb}
\usepackage{amsthm}
\usepackage{amsmath}
\usepackage{amscd} 
\usepackage{pinlabel}
\usepackage[utf8]{inputenc} 
\usepackage{t1enc}
\usepackage[mathscr]{eucal}
\usepackage{indentfirst}
\usepackage{graphicx}
\usepackage{graphics}
\usepackage{pict2e} 
\usepackage{epic}
\numberwithin{equation}{section}
\usepackage[margin=2.9cm]{geometry}
\usepackage{epstopdf} 
\usepackage{hyperref}
\hypersetup{
	colorlinks = true,
	urlcolor   = blue,
	citecolor  = blue,
} 
\usepackage{tikz-cd}
\usepackage{verbatim}
\usepackage{float}
\usepackage{makecell}
\begin{document}
	\newtheorem{de}{Definition}
	\newtheorem{ex}[de]{\emph{Example}}
	\newtheorem{thm}[de]{Theorem}
	\newtheorem{lemma}[de]{Lemma}
	\newtheorem{cor}[de]{Corollary}
	\newtheorem{con}[de]{Conjecture}
    \newtheorem{remark}[de]{Remark}
	\newtheorem{prop}[de]{Proposition}
    \title{A construction of minimal coherent filling pairs}
    \author{Hong Chang and William W. Menasco}

\address{Hong Chang\\
Department of Mathematics\\
University at Buffalo--SUNY\\
Buffalo, NY 14260-2900, USA\\
hchang24@buffalo.edu}

\address{William W. Menasco\\
Department of Mathematics\\
University at Buffalo--SUNY\\
Buffalo, NY 14260-2900, USA\\ menasco@buffalo.edu}
    
	\begin{abstract}
		Let $S_g$ denote the genus $g$ closed orientable surface. A \emph{coherent filling pair} of simple closed curves, $(\alpha,\beta)$ in $S_g$, is a filling pair that has its geometric intersection number equal to the absolute value of its  algebraic intersection number. A \emph{minimally intersecting} filling pair, $(\alpha,\beta)$ in $S_g$, is one whose intersection number is the minimal among all filling pairs of $S_g$. In this paper, we give a simple geometric procedure for constructing minimal intersecting coherent filling pairs on $S_g, \ g \geq 3,$ from the starting point of a coherent filling pair of curves on a torus.  Coherent filling pairs have a natural correspondence to square-tiled surfaces, or  {\em origamis}, and we discuss the origami obtained from the construction.
	\end{abstract}
	\maketitle
	\section{Introduction}
	A simple closed curve on a compact closed surface, $S_g$, of genus $g \geq 2$ is called \emph{essential} if it does not bound a disc. As such, going forward a ``curve in $S_g$'' will mean an essential simple closed curve in $S_g$. Two curves in $S_g$ intersect coherently if all the intersection points have the same orientation provided that the two curves are oriented. Note that it does not depend on the choice of the orientation of the curves.  Two curves are in \emph{minimal position} if the number of intersections of these curves is the minimal within the curves isotopy classes.  It is a simple observation that a coherently intersecting pair are already intersecting minimally within their isotopy classes.  Thus, for coherently intersecting curves this convenience allows us to drop the distinction between working with a curve pair and their isotopy classes.  A pair of curves in $S_g$ is \emph{filling} if over all representatives from their isotopy classes their complement in $S_g$ is a collection of discs.
	
	Let $\alpha, \beta \subset S_g$ be a filling pair.  We call $(\alpha, \beta)$ a minimally intersecting filling pairs if the intersecting number, $i(\alpha,\beta)$, is minimal among all filling pairs on $S_g$.  For Euler characteristics reasons, $i(\alpha, \beta) \geq 2g-1$.  Additionally, for Euler characteristic reasons a minimally intersecting filling pair would have the property that $S_g \setminus (\alpha \cap \beta)$ is a single open disc.  For $g=1$ this low bound is geometrical realizable with a meridian-longitude pair.  For $g=2$, an exhaustive search the finite possibilities for a filling pair intersecting $3$ times establishes that none exist and that at least $4$ intersections is needed.  However, Aougab and Huang showed in \cite{AH} that for all $g \geq 3$ there exists filling pairs of curves whose intersection achieves the $2g-1$ minima.  (More recently, also see \cite{J, N}.) Moreover, the minima can be obtained with $\alpha$ and $\beta$ intersecting coherently shown in \cite{AMN}, so ``minimally intersecting filling coherent pair'' is not an empty set for $g \geq 3$.

 The construction of minimally intersecting filling coherent pairs in \cite{AMN} largely utilizes the algebraic techniques coming from the symmetric groups.  In this paper we give an alternate geometric construction coming from simple cut-and-paste techniques.  Our construction allows one to rapidly construct by hand such filling pairs for any genus.

 \subsection{Coherent filling pairs and origamis.}

 For a filling pair of curves, $\alpha, \beta \subset S_g$, positioned in their isotopy classes so as to have $| \alpha \cap \beta |$ being minimal, the $4$-valent graph, $\alpha \cup \beta = C^1 \subset S_g$, can be thought of as the $1$-skeleton of a $2$-dimensional CW structure on $S_g$.  The dual graph, $\widehat{C}^1 \subset S_g$, is also the $1$-skeleton of a CW structure that has every $2$-cell being a quadrilateral containing a single $0$-cell, an intersection point from $\alpha \cap \beta (\subset C^1)$.  Giving each such quadrilateral $2$-cell a $[0,1] \times [0,1]$ Euclidean square structure, we obtain a square tiling of $S_g$.  If we have $\alpha$ and $\beta$ intersecting coherently then an orientation assignment to these two curves gives a natural way to assign the ``right-side, left-side, bottom, top'' categories to the four boundary $1$-cells of each square tile.  It follows that each $1$-cell of $\widehat{C}^1$ is either a left/right side gluing or a bottom/top gluing of two square tiles.  That is, we have an {\em origami} structure for $S_g$.  Specifically, we have the following result from the literature.

 \begin{thm}[\cite{CJM, J}]
     A coherent filling pair of curves naturally corresponds to an origami on $S_g$.
 \end{thm}

The Euclidean square tiles of an origami gives us a flat geometry except at finitely many branched points---one for each $2$-cell of the CW structure---which correspond to the points where the corners of the tiles are adjoined.  An origami coming from an intersecting coherent filling pair is said to have a single \emph{horizontal cylinder} and a single \emph{vertical cylinder}.  Fig. \ref{origami eg} (Fig. 1 of \cite{AMN}) illustrates such an origami for $S_3$ which happens to be the unique origami up to labeling for genus three \cite{C, AMN}.

	\begin{figure}[h]




\includegraphics[width=.9\linewidth]{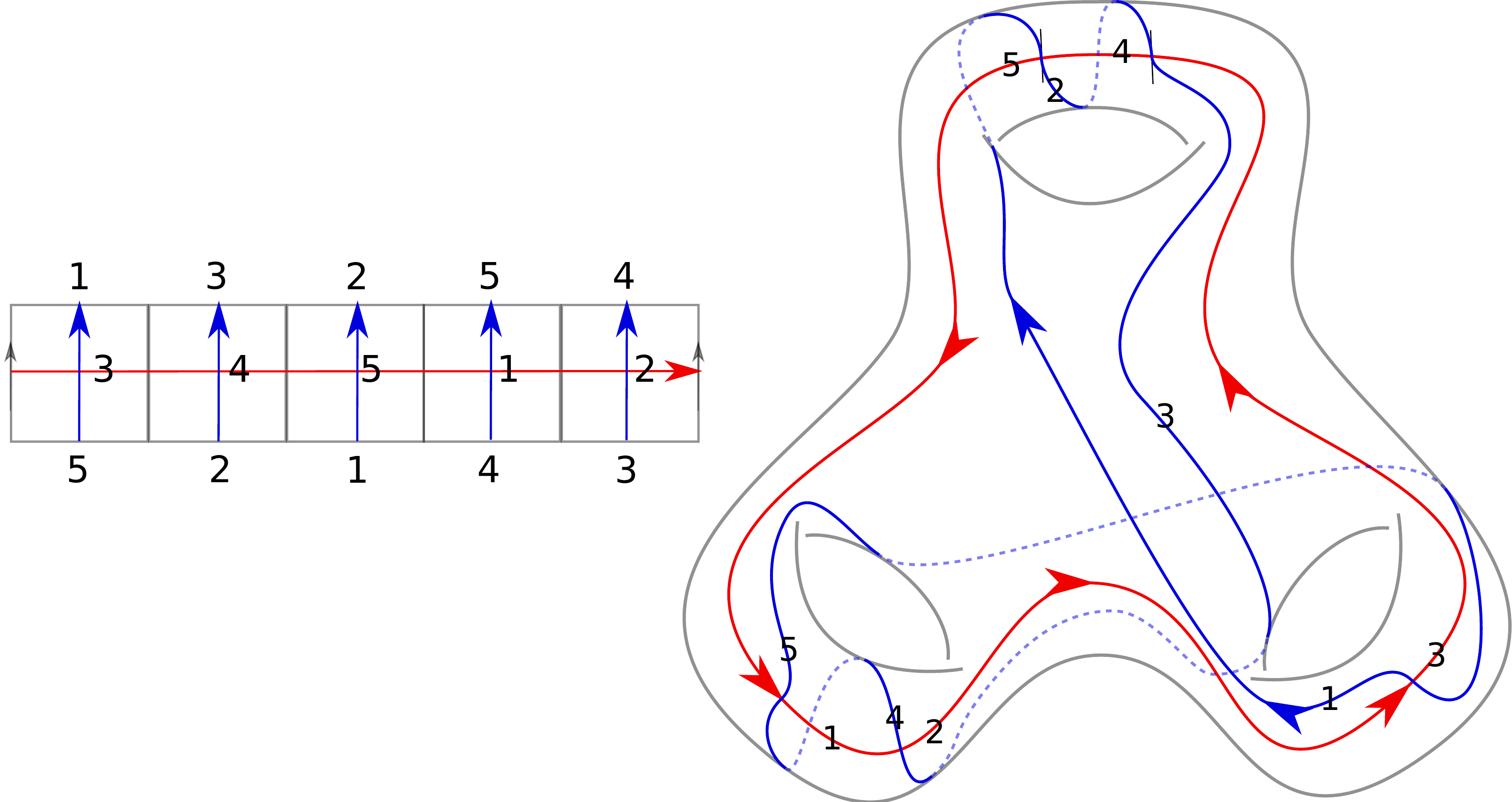}
 
	\vspace{0cm}
	\caption{The left illustrates a filling pair (which is $C^1$), the associated square tiles, and the left/right bottom/top gluing assignment for the origami structure. The right illustrates its geometric realization on $S_3$.  The left's numeric labeling of red/blue edges of $C^1$ correspond to the numeric labeling of the right.  Note that there is exactly one branched point since there is a single component of $S_3 \setminus (\alpha \cup \beta)$.}
	\label{origami eg}
\end{figure}

Flat metrics on $S_g$ coming from the Euclidean square tilings on $S_g$ induces a horizontal foliation on each individual square tile that can be extended to a measured singular foliation $\mathcal{F}$ on $S_g$. A fundamental theorem of Hubbard-Masur associates a unique quadratic differential to this horizontal foliation and, therefore, to the square-tiling \cite{HM}. A origami of $S$ can thus be interpreted as a point in the space of unit-area quadratic differentials.  When $\mathcal{F}$ is orientable, the origami supports the structure of a translation surface, and the associated quadratic differential will be the square of an abelian differential.
 
	\subsection{Outline}
 
	In \S~\ref{intro to surgery}, we introduce the cut-and-paste surgery operations that we will be utilizing.  These surgery operations have the property that, starting with a coherent filling pair of curves on a torus, we will trade an increase in the genus of the surface for a reduction in the number of discs in the complement of the filling pair.  A minimal intersecting filling pair will be realized when the number of disc components is reduced to one.
    
    Previous work on constructing and understanding minimal intersecting filling pairs has focused on the growth of the number of nonequivalent pairs \cite{AH, AMN, AT, C}.  As such, understanding how one might create nonequivalent minimal coherent filling pairs via our surgeries is of interest.  In \S~\ref{tree}, we will use a ``tree-graph'' analysis to understand how our surgery construction can yield nonequivalent filling pairs.  
    
    
    Finally, in \S~\ref{puncture case} we generalize the construction to $S_{g,p}$, oriented surfaces of genus $g$ with $p$ punctures.  In the setting of punctures surfaces the complement of a filling pair of curves is a collection of discs and once punctured discs.

\subsection*{Notation} Throughout this note we will use $S_g$ to denote a closed orientable surface of genus, $g$.  $S_{g,p}$ will denote an closed orientable surface of genus $g$ with $p$ marked or puncture points.  $\partial Y$ denotes the boundary of a compact surface, $Y$.  $|X|$ denotes the cardinality of the set $X$.

 \subsection*{Acknowledgements}  We thank the authors of \cite{AMN} for the use of Fig.~\ref{origami eg}. This work has its genesis in the second author's collaboration with Tarik Aougab and Mark Nieland and he thanks them for numerous discussion on the topic of this note. 
 
	\section{Surgeries on coherent filling pairs on $S_1$} \label{intro to surgery}

 Our strategy for constructing a minimal coherent intersecting filling pair for a genus $g \geq 3$ oriented closed surface is to start with a coherent intersecting filling pair of curves, $(\alpha ,\beta)$, on the torus, $S_1$, that intersect $g$-times.  (Using the ordered $2$-tuple, $\langle m, l \rangle (=\langle {\rm meridian, longitude} \rangle) \in \mathbb{Z} \times \mathbb{Z}$ that specifies curve isotopy classes on $S_1$, the reader might think of $\alpha \subset S_1$ as being the $\langle 0,1 \rangle$ curve and $\beta \subset S_1$ being the $\langle g,1 \rangle$ curve.)  Our construction requires that we consider two cases, when $g$ is odd and when $g$ is even.  For the odd case, through a simply surgery operation on the graph, $\alpha \cup \beta \subset S_1 $, we will add in $g-1$ new vertices.  Initially, $S_1 \setminus (\alpha \cup \beta)$ has $g$ disc components.  Our simple surgery on $\beta$ will trade disc components for genus---each surgery decreases the number of disc components by one while increasing the genus of the resulting surface by one.  The genus of the final surface will be $2g-1$ and there will be exactly one disc component in the complement of the resulting filling pair, $(\alpha, \beta^\prime)$, for $S_{2g-1}$.  For the case when $g$ is even, will need one additional simply surgery to trade disc-for-genus all the way up to $2g-1$ genus. 
 
	\subsection{Two simple surgeries on filling pairs} \label{band}
	
	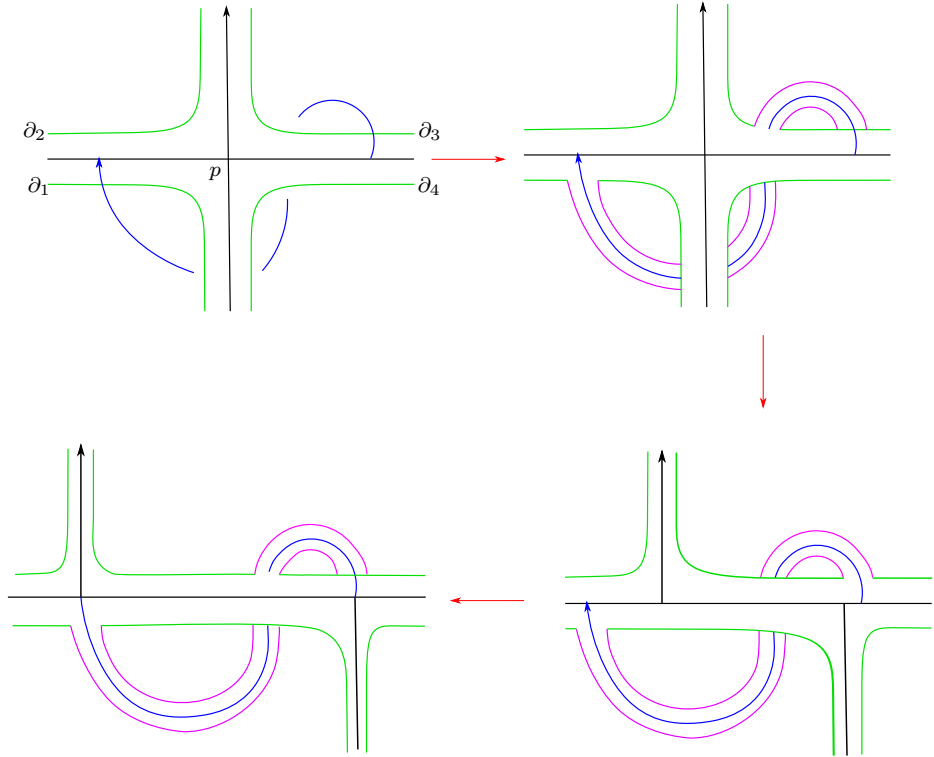
\begin{figure}[h]

 \labellist
\tiny

\pinlabel $\partial_1$ [tl] at 8 328
\pinlabel $p$ [tl] at 109 334
\pinlabel $\partial_2$ [t] at 15 360
\pinlabel $\partial_4$ [tl] at 225 328
\pinlabel $\partial_3$ [t] at 235 360

\endlabellist

\includegraphics[width=.8\linewidth]{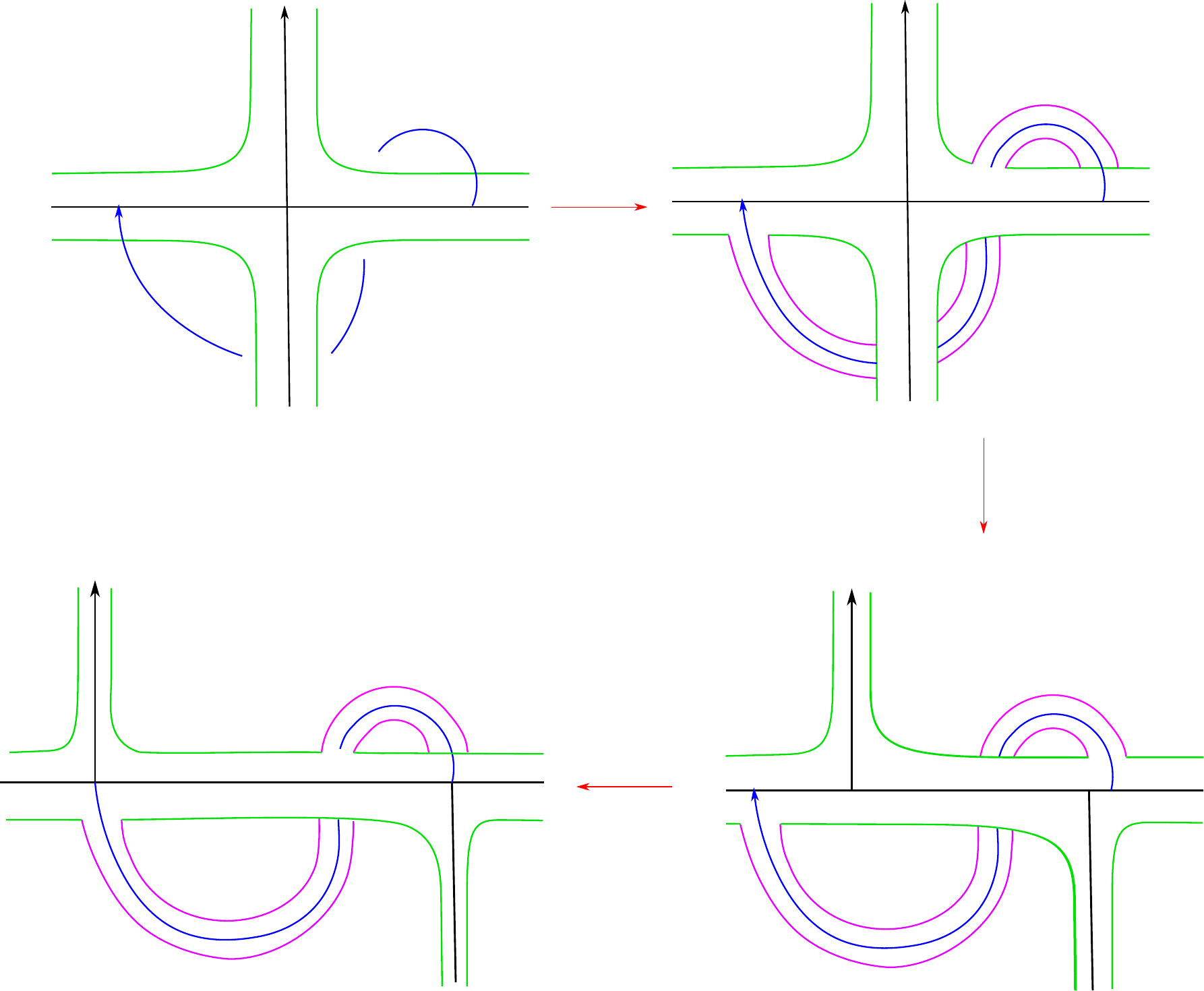}
 
	\vspace{0cm}
	\caption{The first illustration of the sequence shows the extended core of the band-to-be-added.  Its endpoints are on $\alpha$.  The second illustration shows the added band. The third and fourth illustration in the sequence show how the ``shear'' the intersection point in $\alpha \cap \beta$ and adjoin, or ``splice'', to the endpoints of the extended core of the band.  The salient feature is the $\partial_1$ and $\partial_3$ are band connected.}
	\label{oneband}
\end{figure}

Let $(\alpha, \beta)$ be a filling pair for $S_{g \geq 1}$ and consider a closed regular neighborhood, $\mathbf{N}$, of the graph, $\alpha \cup \beta \subset S_g$.  We assume that $\alpha$ and $\beta$ are positioned so as to intersect minimally within the isotopy class of, say, $\beta$.  As such, each boundary component of $\partial \mathbf{N}$ bounds a disc in $S_g$.  In particular, we focus on any small neighborhood, $\nu(p) \subset \mathbf{N}$, around $p \in \alpha \cap \beta$, one of the $4$-valent intersection points---the first illustration in the sequence of Fig.~\ref{oneband} depicts $\nu(p)$.  $\nu(p)$ will have segment portions of four boundary components, $\partial_1$, $\partial_2$, $\partial_3$ and $\partial_4$ as shown in first illustration in the sequence of Fig.~\ref{oneband}.  We remark that some of $\partial_i{\rm 's}$ may be the same component of $\partial N$.  Taking $\alpha$ near $p$ as a west/east axis and $\beta$ as a north/south axis, the four boundary segment are positioned so that $\partial_1$ is Southwest (SW), $\partial_2$ is NW, $\partial_3$ is NE, and $\partial_4$ is SE.

{\bf The single $1$-handle surgery}---We now glue to the sub-surface, $ \mathbf{N}$, a $1$-handle, $B (\cong [0,1] \times [0,1])$, that is attached to $\partial_1$ (SW) and $\partial_3$ (NE).

    Referring to the second illustration in the sequence in Fig.~\ref{oneband}, we take an arc, $\gamma$, to be the extended core of the attached $B$. The salient feature is that $\gamma$ is attached to the south side (north side) of the west portion (east portion) of $\alpha \cap \nu$.  Then $\alpha \cup \beta \cup \gamma$ will be a graph in $\mathbf{N} \cup B$ that has some number of $4$-valent vertices---same number as $|\alpha \cap \beta|$---and two $3$-valent vertices---the two endpoints of $ \gamma$.

    The third illustration in the sequence in Fig.~\ref{oneband} show a {\em shearing} of $\beta$ at the point $p$, creating two new $3$-valent vertices.  The reader should observe that we now have four $3$-valent vertices in succession on $\alpha$. The fourth illustration shows how these four $3$-valence vertices are realigned and {\em spliced} to create two new $4$-valence vertices and a new $\beta^\prime$.  The key feature of the final fourth illustration is that the orientation at the two intersections of $\alpha \cap \beta^\prime$ created by this splice is consistent with the original orientation intersection point, $\nu \cap (\alpha \cap \beta)$---crossing $\alpha$ south to north.

    We observe that if $\partial_1 \not= \partial_3$ then the $\partial(\mathbf{N} \cup B)$ has one less boundary component and the genus of $\mathbf{N} \cup B$ is increased by one.  Moreover, the curve pair $(\alpha, \beta^\prime)$, will be a filling pair in the surface obtains by capping off each component of $\partial(\mathbf{N} \cup B)$ with a disc, i.e. $S_{g+1}$.  Additionally, $|\alpha \cap \beta| +1 = |\alpha \cap \beta^\prime|$.

    The surgery sequence obviously is generalized by rotation and reflection. 

    We will refer back to this {\em shear and splice} construction numerous times in this note.
	
	\begin{figure}[h]

 \labellist
\tiny

\pinlabel $\partial_1$ [tl] at -20 85
\pinlabel $p$ [tl] at 140 89
\pinlabel $\partial_2$ [t] at -12 110
\pinlabel $\partial_4$ [tl] at 240 85
\pinlabel $\partial_3$ [t] at 250 110

\endlabellist

\includegraphics[width=.8\linewidth]{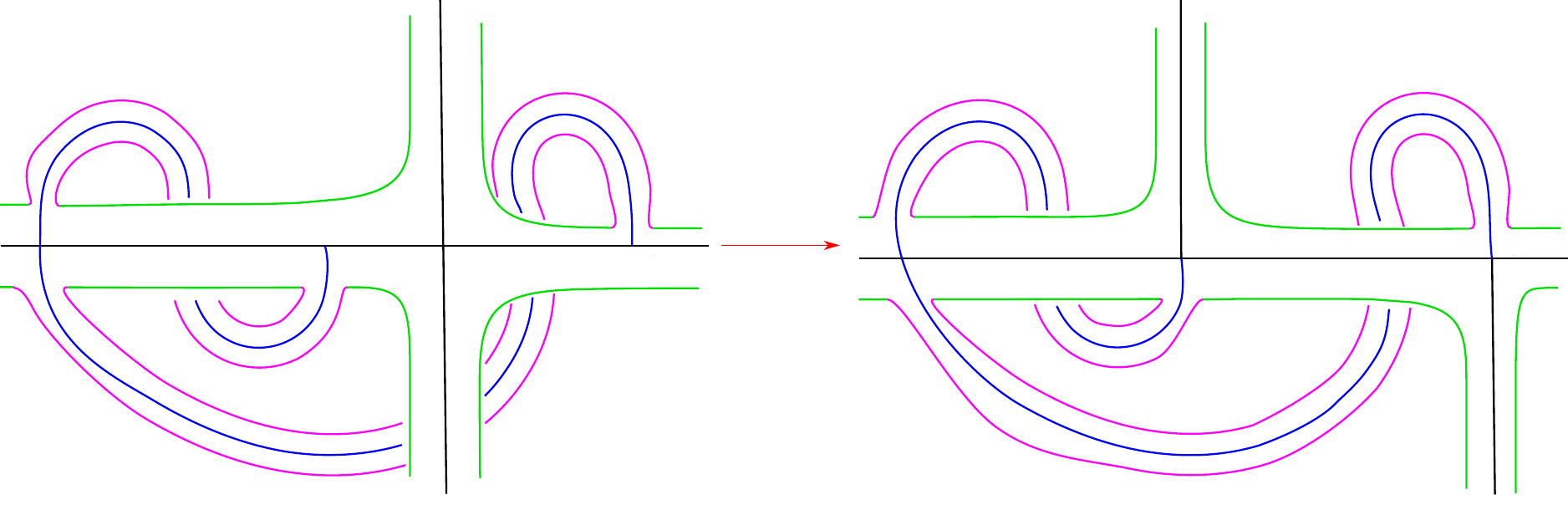}
	\caption{Banding three boundary components with one arc}
	\label{twoband}

\end{figure}
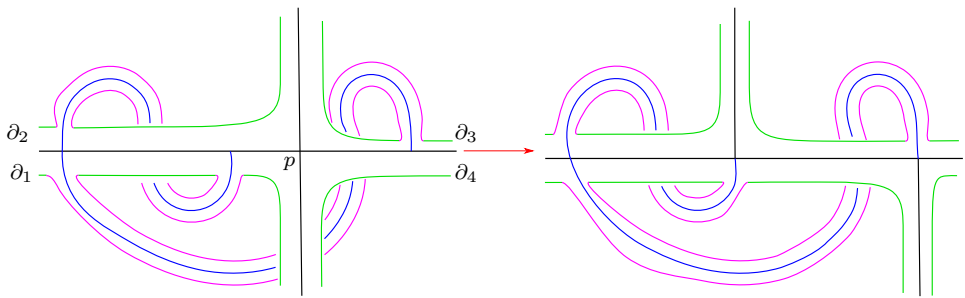

    {\bf The two $1$-handle surgery}---For this surgery we refer the reader to Fig.~\ref{twoband} on how we will alter the initial $\nu(p)$ neighborhood.  Specifically, we glue in two $1$-handles: a $1$-handle, $B_{NW/SW}$, that is attached to $\partial_2$ (NW) and $\partial_1$ (SW); and, a $1$-handle, $B_{NW/SE}$ attached to, again, $\partial_1$ (SW) and $\partial_3$ (NE).  Next, we take a core arc of each $1$-handle and extend them into $\nu(p)$ so as to create a single arc, $\gamma$, that is attached to $\alpha$ on the north (south) side of the west (east) portion in $\nu(p)$.  The blue arc in the left illustration of Fig.~\ref{oneband} corresponds to $\gamma$.  Note that at this stage $\alpha \cup \beta \cup \gamma$ is a graph in $\mathbf{N} \cup B_{NW/SW} \cup B_{SW/NE}$ having $|\alpha \cap \beta| +1$ $4$-valent vertices and two $3$-valent vertices.

    Finally, we shear $\beta$ at the point $p \in \alpha \cap \beta$ to create two $3$-valent vertices.  As with our first surgery, we will then have four $3$-valent vertices in succession on $\alpha$.  The right illustration of Fig.~\ref{oneband} shows the realignment of these four vertices creating two new $4$-valent vertices and a new $\beta^\prime$ curve by splicing into $\beta$ the extended core arc.  As with our first surgery, the two new vertices of $\beta^\prime$ are intersections with $\alpha$ that are consistent with the manner of intersection of our original point $p$---crossing $\alpha$ south to north.  Thus, again we have a shear and splice construction, going from $\beta$ to $\beta^\prime$.

    If we assume that $\partial_1, \partial_2, \partial_3$ are all distinct boundary curves of $\mathbf{N}$ then $|\partial(\mathbf{N} \cap B_{NW/SW} \cap B_{SW/NE})| = |\partial \mathbf{N}| -2$. Thus, the curve pair $(\alpha, \beta^\prime)$, will be a filling pair in the surface obtains by capping off each component of $\partial(\mathbf{N} \cap B_{NW/SW} \cap B_{SW/NE}))$ with a disc, i.e. $S_{g+2}$.  Additionally, $|\alpha \cap \beta| +2 = |\alpha \cap \beta^\prime|$

    Finally, both the surgery sequences are generalized by rotation and reflection. 
	\subsection{Constructing minimal coherent filling pairs.}
    \label{}
	With our two surgeries in hand we are now in a position to construct minimal coherent intersecting filling pairs for genus, $g \geq 3$.

	As stated at the beginning of \S~\ref{intro to surgery}, we start with a filling pair on $S_1$ that intersect $g$-times.  Again, $\alpha$ is a $\langle 0, 1 \rangle$ curve and $\beta$ as a $\langle g,1 \rangle$ curve.  We give an orientation to $\alpha$ and label the $g$ intersection points, $\{p_1, p_2, \cdots , p_g\} = \alpha \cap \beta$, such that the cyclic order of the points on $\alpha$ correspondence to the cyclic order given by indices of the $\nu(p_i)$-labels.  Next we orient $\beta$ similarly---traversing $\beta$, $mod(g)$ the $i^{\rm th}$ intersection point is $p_i$.

    Our construction requires that we consider the cases when $g$ is odd and even separately.  The top illustration in Fig.~\ref{4} has $g = 3$ and is representative of the cases having $g$ odd.  The bottom illustration of Fig.~\ref{4} has $g=4$ and is representative of the cases having $g$ even.

    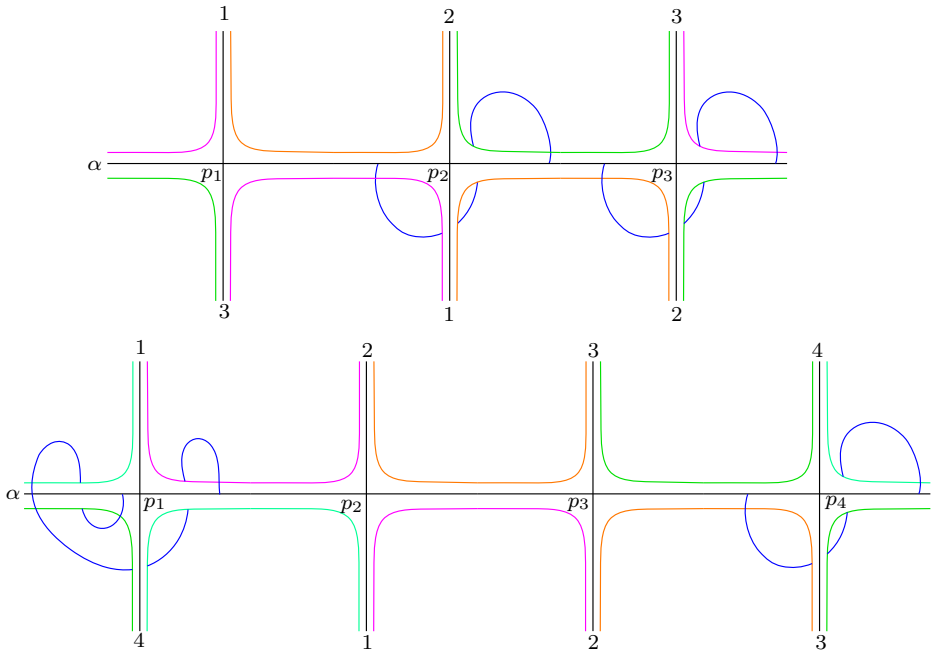
\begin{figure}[h]

 \labellist
\tiny

\pinlabel $\alpha$ [t] at 45 278

\pinlabel $p_1$ [tl] at 103 272
\pinlabel $p_2$ [t] at 243 273
\pinlabel $p_3$ [t] at 372 273

\pinlabel $1$ [tl] at 113 366
\pinlabel $2$ [t] at 249 366
\pinlabel $3$ [t] at 380 366

\pinlabel $3$ [tl] at 113 193
\pinlabel $1$ [t] at 249 193
\pinlabel $2$ [t] at 380 193

\pinlabel $\alpha$ [t] at -2 87

\pinlabel $p_1$ [tl] at 70 81
\pinlabel $p_2$ [t] at 193 81
\pinlabel $p_3$ [tl] at 314 81
\pinlabel $p_4$ [t] at 472 83

\pinlabel $1$ [tl] at 65 172
\pinlabel $2$ [t] at 202 172
\pinlabel $3$ [t] at 332 172
\pinlabel $4$ [t] at 461 172

\pinlabel $4$ [tl] at 64 2
\pinlabel $1$ [t] at 202 2
\pinlabel $2$ [t] at 332 2
\pinlabel $3$ [t] at 463 2

\endlabellist

    \includegraphics[width=.8\linewidth]{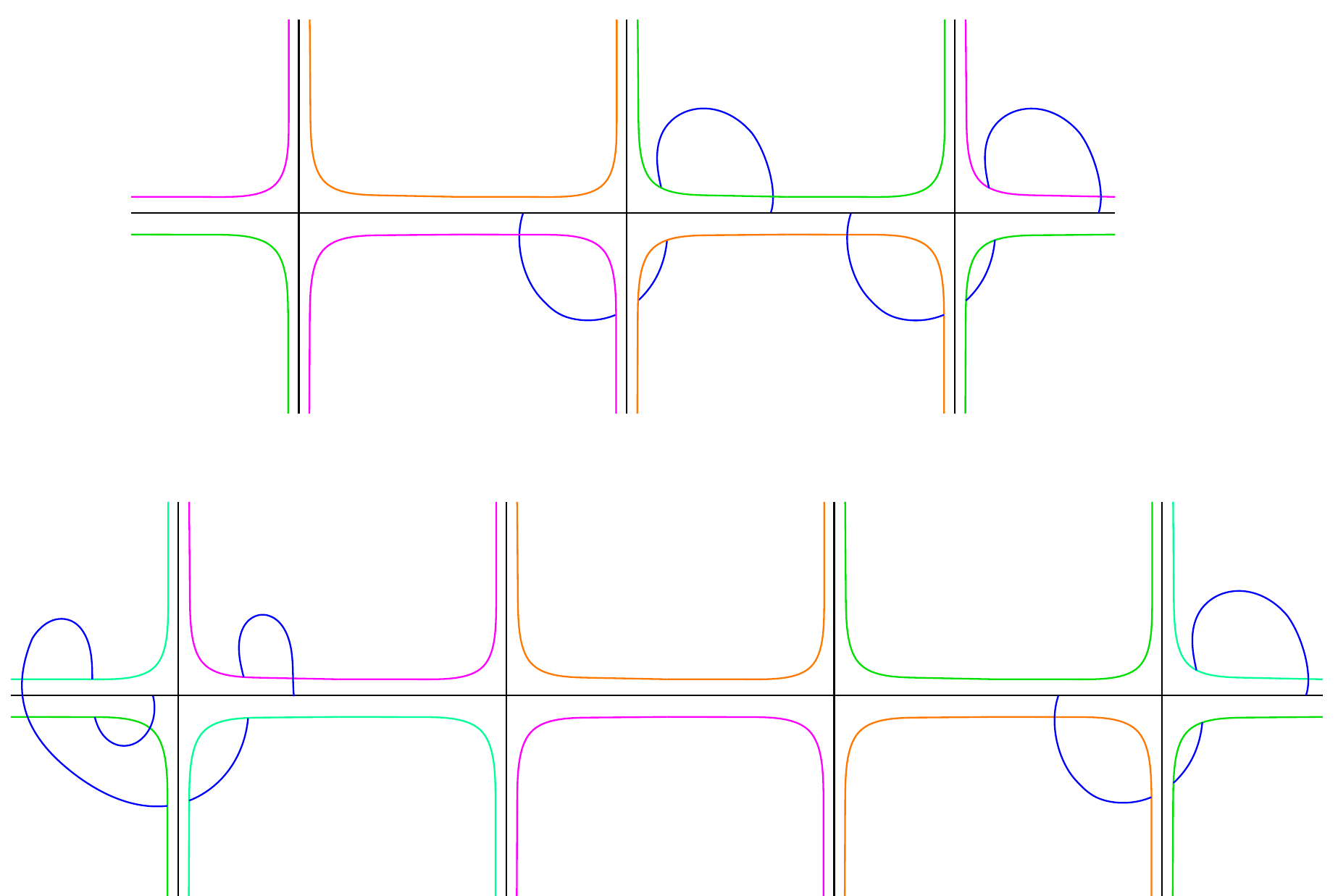}
		\caption{The top illustration has $g = 3$ and is representative of the odd case. The bottom illustration has $g=4$ and is representative of the even case. The horizontal segments in each has the right/left endpoints identified and corresponds to the $\alpha$ curve.  The labels on the endpoints of the vertical segments correspond to the identification of their endpoints so as to form the $\beta$ curve.}
		\label{4}
	\end{figure}

 Referring to Fig.~\ref{4}, it is convenient to represent the $\alpha$ curve by a horizontal line segment which has its left and right endpoints identified.  Then we can represent the $\beta$ curve by $g$ vertical line segment, each one of which intersects our $\alpha$ representation once at its midpoint.  Assigning labels---$1$ through $g$, left to right---to the top endpoints of our $g$ vertical segments and labels, $g$ then $1$ through $g-1$, to the bottom ends of the vertical segments, we realize $\beta$ by a gluing that matches the top endpoint labels with the bottom endpoint labels.
 
 It is also helpful to assign labels, $p_1$ through $p_g$, to the points of intersection of the horizontal $\alpha$ segment with the vertical segments---$p_i$ will be in the vertical segment have $i$ as a top endpoint label.  Next, when we consider a regular neighborhood, $\mathbf{N}$ of $\alpha \cup \beta \subset S_1$, near $p_i$ we have the four ``compass'' boundary curves, $NE_i, NW_i , SW_i , SE_i$, where, due to indexing scheme for connecting the labels of the vertical segments, $NE_i = NW_{i+1}, SE_{i} = SW_{i+1}, NW_i = SE_i$  To help to reader with this identification in Fig.~\ref{4} we have distinguished the components of $\partial \mathbf{N}$ by a color assignment.  The reader should observe that $|\partial \mathbf{N}| =g$.

The key issue is when a ``$1$-handle attaching scheme'' to $\partial \mathbf{N}$ results in a surface with one boundary.  (In \S~\ref{treecon} we will expand  and make precise our notion of ``scheme''.) To that end we define an {\em attaching graph or A-graph}, $G$.  The vertices of $G$ are the components of $\partial \mathbf{N}$.  And, two vertices share an edge if they share the attaching ends of a specified $1$-handle.  Then $G$ will have $g$ vertices and $g-1$ edges.  We then have the following lemma.

	\begin{lemma}
 \label{lemma:tree}
        Given an attaching scheme of $(g -1)$ $1$-handles to $\partial \mathbf{N}$, the resulting surface will have exactly one boundary component if and only if the A-graph, $G$, is a connected tree.
	\end{lemma}
 \vspace{5pt}

 The proof of the following lemma will be delayed until \S~\ref{tree}.

 We now give two schemes---one for $g$ odd and one for $g$ even---attaching $1$-handles to $\mathbf{N}$, both utilize the $1$-handle surgeries of \S~\ref{band}.

\noindent
{\bf Case where $g$ is odd.} In a neighborhood of each intersection point, $p_i, \ 2 \leq i \leq g$, we perform a single $1$-handle surgery attached to ${SW}_i$ to ${NE}_i$ for $2 \leq i \leq g$.  It is readily observed that the graph, $G$, is a linear tree.  (The reader may wish to consult the top of Fig.~\ref{4}.)  Thus, by Lemma~\ref{lemma:tree} the resulting surface has one boundary component and is of genus $(2g-1)$.

As previously observed, the resulting filling pair will still have coherent intersection. $\square$
\vspace{5pt}

\noindent
{\bf Case where $g$ is even.}  In a neighborhood of $p_i$ we perform a two $1$-handel surgery: attaching a $1$-handle between ${NW}_1$ and ${SW}_1$; and, ${NW}_1$ and ${SE}_1$. Then, in a neighborhood of each intersection point, $p_i, \ 4 \leq i \leq g$, we perform a single $1$-handle surgery attached to ${SW}_i$ to ${NE}_i$ for $4 \leq i \leq g$.  Again, it is readily observed that the associated graph, $G$, is a linear tree.  (The reader may wish to consult the bottom of Fig.~\ref{4}.)  And, Lemma~\ref{lemma:tree} again gives us that the resulting surface has one boundary component and is of genus $(2g-1)$.

And again, the resulting filling pair will still have coherent intersection.
$\square$
\vspace{5pt}

Based upon the above two surgery schemes we can state the following result.

\begin{thm}
 \label{theorem: scheme}
		For genus $g\geq 3$, we can create minimal coherent filling pairs utilizing the two $1$-handle surgeries of \S~\ref{band}.
	\end{thm}

By now the reader may have realized that there are other choices one may make for attaching $1$-handles, shearing vertices and slicing in the extended handles cores so as to obtain a single boundary curve and a new $\beta^\prime$.  In the next sections we investigate other such choices.
 
 

 \begin{remark}
{\rm For $g=2$, if we attempt to attach a single $1$-handles to $\partial \mathbf{N}$ using the single $1$-handle surgery we ``run out of room''.  That is, the associated graph, $G$, will not be a connected tree since both points,  $p_1 , p_2 \subset \alpha \cup \beta \subset \mathbf{N}$, are adjacent to just the two boundary curves of $\partial \mathbf{N}$.  Thus, we cannot realize a filling pair such that $\alpha \cap \beta = 2 \cdot g-1 = 2\cdot 2 - 1 =3$.  This ``failure to construct'' is consistent with the fact the for genus $2$ we need $|\alpha \cap \beta| = 4$.}
 \end{remark}

\section{$1$-handle attaching schemes and filling pair} \label{tree}

We now supply the proof of our previously used lemma.

\begin{proof}[Proof of Lemma~\ref{tree}.]

First, we will assume that the resulting surface has exactly one boundary component.  We will argue that $G$ must be a tree.

We observe that our definition of the graph $G$ is really independent of the filling pair and only dependent on the surface type of $\mathbf{N}$.  That is, $\mathbf{N}$ is homeomophic to $S_{1,0,g}$, a compact surface of genus one having no puncture (or marked points) and $g$ boundary components.  By attaching a $1$-handle to any surface with boundary we either increase by one or decrease by one the number of boundary components of the resulting surface.  To do the former (latter), both ends of the $1$-handle must be attached to the same (different) boundary component(s).

With the above in mind, we take a $\hat{g}$ to be of minimal value for which the theorem is not true.  Then for $\mathbf{N} \cong S_{1,o,\hat{g}}$, there is an attaching scheme of $(\hat{g}-1)$ handle on the $\hat{g}$ boundary components that produce a single boundary curve, but the associated A-graph is not a tree.

We next take a maximal sub-collection of handles that results in an A-graph, $G^\prime$, having each component of $G^\prime$ is a tree.  (By assumption this sub-collection has fewer than $(\hat{g} -1)$ handles.)  The cardinality of our sub-collection of handles is $(\hat{g} - |G^\prime|)$.  And, there are $|G^\prime| - 1$ remaining handles to attach.  Moreover, each component accounts for one boundary component of the resulting sub-surface, i.e. there are $|G^\prime|$ boundary component.

If we now attach a one of the handles not in our maximal sub-collection, it must result in a component of our A-graph being not a tree.  This implies that both ends of this handle are attached to the same boundary component.  But, this is not possible.  We are in the situation where we have a surface with $|G^\prime| ( < \hat{g})$ boundary curves and $|G^\prime| -1 $ handles to be attach resulting in a single boundary component.  By the assumption that $\hat{g}$ is minimal for realizing a counterexample, we have a contradiction.

For the other direction of our theorem we proceeding inductively.  If there are initially two components of $\partial \mathbf{N}$, attaching a $1$-handle between them will produce a single boundary component and the associated $G$ is a tree.  Now suppose we can attach  


Now consider the associated graph, $G$, coming from an attachment scheme of $(g-1)$ $1$-handles to the $g$ boundary components of $\mathbf{N}$.  And, assume that $G$ is a connected tree.  We use $\mathbf{N}^\prime$ to denote the resulting surface and we observe that $\mathbf{N}$ is naturally seen as a sub-surface in $\mathbf{N}^\prime$.  Moreover, we can obtain $\mathbf{N}$ from $\mathbf{N}^\prime$ by deleting the open sets in $\mathbf{N}^\prime$ that correspond to the ``interior'' of the $1$-handles---homeomorphically equivalent to $(0,1) \times [0,1]$.

We then have a similar behavior to that described in the first half of our argument.  By deleting the interior of a $1$-handle to any surface with boundary we either increase or decrease by one the number of boundary components of the resulting surface.  To do the former (latter), both components of $\{1 {\rm-handle}\} \cap \partial \mathbf{N}^\prime$ must on the same (different) boundary component(s) of $\mathbf{N}^\prime$.  Since the deletion of $(g-1)$ interiors of $1$-handles in $\mathbf{N}^\prime$ produces a surface with $g$ boundary components, by $G$ being connected---every boundary of $\mathbf{N}$ has at least one $1$-handle attached---we conclude that $|\partial \mathbf{N}^\prime|=1$.
\end{proof}
	
		
As previously observed, the $1$-handle attaching scheme of Theorem~\ref{theorem: scheme} is not unique in that one can readily construct other $1$-handle attaching schemes whose associated graph $G$ is a tree.  In Fig.~\ref{treecon} we offer such an example.

	\begin{figure}[h]

  \labellist
\tiny

\pinlabel $\gamma$ [t] at 59 160





\pinlabel $p_1$ [tl] at 17 120
\pinlabel $p_2$ [t] at 65 120
\pinlabel $p_3$ [tl] at 102 120
\pinlabel $p_4$ [t] at 150 120
\pinlabel $p_5$ [tl] at 187 120
\pinlabel $p_6$ [t] at 236 120



\endlabellist

\includegraphics[width=.8\linewidth]{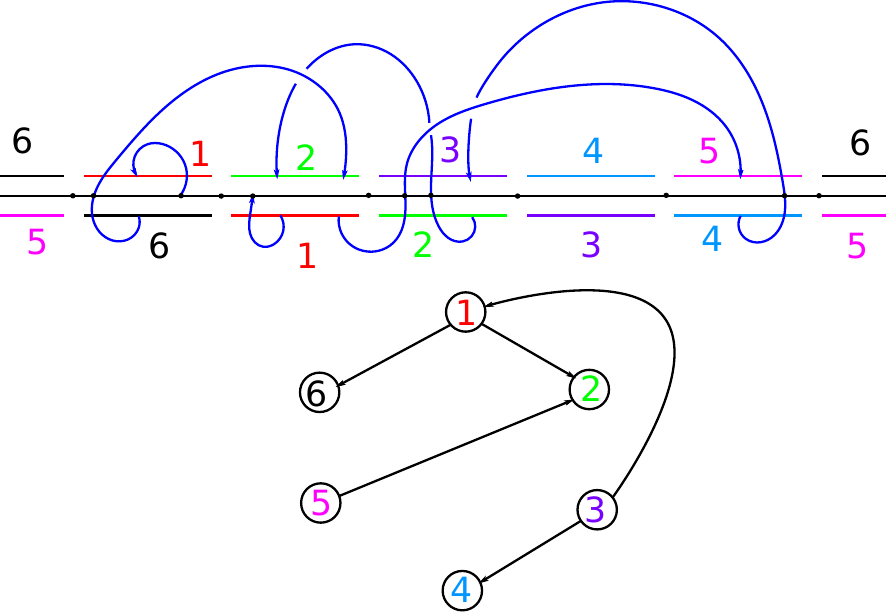}
		\caption{The arc, $\gamma$, contains the extended core of five $1$-handles.  There is only one shear and splice which is at $p_2$.  It splices $\gamma$ into $\beta$ to produce  $\beta^\prime$.  The lower illustration depicts the associated graph, $G$, which is a tree.}
		\label{treecon}
	\end{figure}
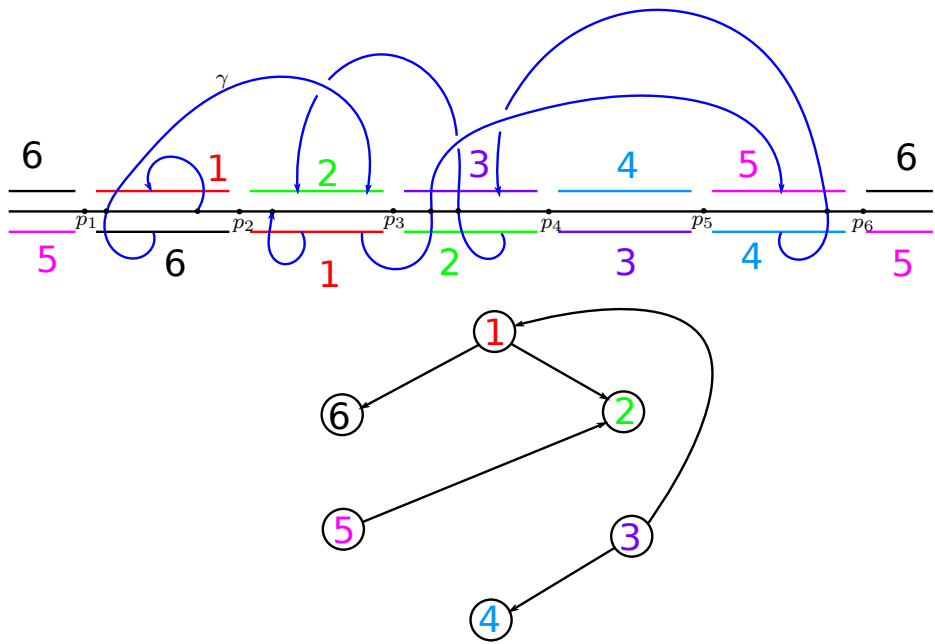

\begin{ex}[An attaching scheme for $S_{11}$.]{\rm 
Initially, the $\beta$ curve is a $(6,1)$ curve on $S_1$ with $\alpha$ again being the $(0,1)$ curve.  As before, we will have $\{p_1, \cdots, p_6 \} = \alpha \cap \beta$ which we indicate in the top illustration of Fig.~\ref{treecon}.  To reduce the clutter we do not depict the vertical arcs associated with $\beta$.  Finally, we double-label the boundary components of $\partial \mathbf{N}$ with numeric-colored labels.

We depict an oriented blue arc, $\gamma$, that has endpoints in $\alpha$ to the left and right of the point, $p_2( \subset \alpha \cap \beta)$.  We will use $\gamma$ for a scheme of attaching five $1$-handles to the components of $\partial \mathbf{N}$.  Specifically, $\gamma$ is the union of five extended $1$-handled cores.  These five extended cores have their endpoints in $\alpha$: one to the right of $p_1$; one to the right of $p_3$; one to the left of $p_6$; and, the two endpoints that are to the left and right of $p_2$.

The last two listed points, the endpoints of $\gamma$ are of particular interest.  They are positioned on $\alpha$ so that we can splice $\gamma$ into $\beta$ by a shearing of $\beta$ at $p_2$ followed by a reconnecting of endpoints as previously depicted in Fig.'s~\ref{oneband} \& \ref{twoband}.  The resulting curve will be our new $\beta^\prime$.  Returning to the orientation assignment of $\gamma$, it is consistent with the orientation of $\beta$ that we have been assigning---edges of $\beta$ are depicted as coming into $\alpha$ from below and going out of $\alpha$ from above.  Thus, the new $\beta^\prime$ will have coherent intersection with $\alpha$.

The bottom illustration of Fig.~\ref{treecon} depicts the associated A-graph, $G$.  (The validity of $G$ we leave it to the reader to check.)  Since $G$ is a tree by our Lemma~\ref{lemma:tree} we conclude that this scheme for attaching $1$-handles yields a minimal coherent filling pair for $S_{11}$.  Moreover, the extended cores of the $1$-handles inherent an orientation from $\gamma$ that results in giving each edge of $G$ an orientation, e.g. the core of the $1$-handle goes from the $1$/red boundary curve to the $6$/black boundary curve giving us a edge in $G$ going from the $1$/red vertex to the $6$/black vertex.}
$\diamond$
\end{ex}

From this example we see that an attaching scheme for a collection of $1$-handles can be described by specifying a disjoint collection of oriented arcs, $\{\gamma_1, \cdots , \gamma_n\}$, that have their endpoints on our initial $\alpha$ curve---a $ \langle 0,1 \rangle$ curve on $S_1$---and intersect $\alpha$ is a coherent manner that is consistent with that of $\beta$---a $\langle g,1 \rangle$ curve on $S_1$.  Then, $1$-handles are attached to $\mathbf{N}$ so as to have their extended cores equal $\cup_{1 \leq i \leq n} \gamma_i$.  Each $\gamma_i$ satisfies the following conditions.  (We continue our appeal to the setup: $\alpha$ and $\beta$  curves in $S_1$ coherently intersect and $\mathbf{N}$ is a regular neighborhood $\alpha \cup \beta$.  We also visually depict the setup in the same manner: $\alpha$ a single horizontal arc with left/right endpoints identified; and, $\beta$ as $g$ vertical segments oriented pointing bottom-to-top.)

\begin{itemize}
    \item[1.]    For $\partial \gamma_i$ there exists an intersection point $p \in \alpha \cap \beta$ such that on $\alpha$ these two endpoints are to the immediate left/right of $p$.
    \item[2.]  $\gamma_i$ intersects $\alpha$ in a coherent manner. 
    \item[3.] $\gamma_i$ is attached to $\alpha$ and oriented such that a shear and splice operation at the point $p$ (previous condition) yields a consistently oriented curve coherently intersecting $\alpha$.
\end{itemize}
A collection of $\gamma$ arcs satisfying the above three conditions are said to be a {\em $1$-handle attaching scheme}.

Performing the shear and splice operation for each $\gamma$ arc of a $1$-handle attaching scheme will yield a curve pair, $(\alpha, \beta^\prime)$, for some oriented some closed surface.  Additionally, a corresponding A-graph, $G$, can be constructed.

We have the following theorem whose proof is now self-evident.

\begin{thm}
    Let $\{\gamma_1, \cdots , \gamma_n\}$ be a $1$-handle attaching scheme.  Suppose $$|\cup^n_1 \gamma_i \cap \alpha| -n = g-1 . $$
    Then the resulting curve pair is a minimal coherent filling pair for a $S_g$ if and only if the A-graph $G$ is a connected tree.
\end{thm}
	
	\section{Cases with punctures}
 \label{puncture case}

 We now extend our minimal coherent filling pair construction to orientable finite type surfaces, $S_{g,p}$, where the genus is $g\ge 3$ and there are $p(>0)$, punctures (or marked) points.  A pair of curves, $\bar{\alpha}, \bar{\beta} \subset S_{g,p}$, is {\em filling} if, when $\bar{\alpha}$ and $\bar{\beta}$ are positioned to intersect minimally within their isotopy classes, $S_{g,p} \setminus (\bar{\alpha} \cup \bar{\beta})$ is a collection of discs and once punctured discs.  By an Euler characteristic argument, the minimal number of intersections needed for a pair of curves to fill is $2g + p -2$ \cite{J}.  If $(\bar{\alpha} ,\bar{\beta})$ is a minimal filling pair then $|S_{g,p} \setminus (\bar{\alpha} \cup \bar{\beta})| = p$.  Alternatively, if we consider a regular neighborhood, $\mathbf{N}(\bar{\alpha} \cup \bar{\beta}) \subset S_{g,o,p}$ we would have $|\partial \mathbf{N}| = p$.
 
We slightly modify our initial setup of $\alpha, \beta \subset S_1$ by requiring $|\alpha \cap \beta| = g + p -1$.  Thus, $\beta$ is now a $\langle (g + p -1) ,1 \rangle$ curve on $S_1$ while $\alpha$ is still a $\langle 0,1 \rangle$ curve.  This implies that $S_1 \setminus (\alpha \cup \beta)$ has $(g+p-1)$ discs components.  We again denote a regular neighborhood of $\alpha \cup \beta$ by $\mathbf{N} (\subset S_1)$.

We define a collection of $\gamma$ arcs, $\{\gamma_1 , \cdots , \gamma_n\}$, as being a $1$-handle attaching scheme in exactly the same manner as that of \S~\ref{tree}.  As such we can consider the A-graph, $G$, of an attaching scheme.  If $|G| = p$ with each connected sub-graph component being a tree where exactly one of its vertices corresponds to a $\partial_i$ boundary component of $\mathbf{N}$, then the resulting curve pair, $(\alpha , \beta^\prime)$, will be filling.  And additionally, if $$| \cup_1^n \gamma_i \cap \alpha| - n = g  -1 , $$ the resulting pair, $(\alpha , \beta^\prime)$, will be minimal, i.e. $(g + p -1) + (g-1) = 2g + p -2$.  Since our third condition for an attaching scheme requires that the shear and splice operation produce a $\beta^\prime$ that coherently intersects $\alpha$, we have our conditions for constructing minimal coherent filling pairs for $S_{g,p}, g \geq 3, p>0$.

\begin{thm}[Also see \cite{J}, Theorem 1.3.]
    \label{theorem:finite type}
    Minimal coherent filling pairs for $S_{g,p}, g \geq 3, p>0$, exist for all such $g$ and $p$.
\end{thm}

\begin{proof}
    Utilizing only the two $1$-handle surgeries of \S~\ref{intro to surgery}, we take any $1$-handle attaching scheme of $\gamma$ arcs that yields a minimal coherent filling pair for $S_{g+p}$.  We throw away any $p-1$ edges of the associated tree graph, $G$, to produce a non-connected graph, $G^\prime$, which has $p$ sub-graphs.  We restrict our choices of discarded edges to those that correspond to single $1$-handle surgery and the $B_{NW/SW}$ handle of the two $1$-handle surgery.  The reader should observe that throwing away the $B_{NW/SW}$ alters the two $1$-handle surgery to a single $1$-handle surgery.

    Thus, by construction there is a sub-collection of our original collection of $\gamma$ arcs that yield an attaching scheme that producing a surface of genus, $g$, with $p$ boundary components.  Moreover, $G^\prime$ is the associated A-graph and this surface will be a regular neighborhood of the resulting coherent filling pair four valent graph, $\alpha \cup \beta^\prime$.  Now by capping off each of the $p$ boundary components with a once punctured disc, we obtain $S_{g,p}$.  Necessarily, $|\alpha \cap \beta^\prime |= 2g + p -2$.

    One last subtle observation.  We can designate one arbitrary vertex from each component of $G^\prime$ as corresponding to a one of the $p$ designated boundary components, $\partial_i$. Placing a puncture in each associated disc of $S_1 \setminus (\alpha \cup \beta)$, we can then have this $S_{1,p}$ with its coherent filling pair, $(\alpha, \beta) $, as the initial starting setup.  The previous sub-collection of $\gamma$ arcs will then be an attaching scheme that yields $(\alpha , \beta^\prime)$ as a minimal coherent filling pair in a $S_{g,p}$ surface.  
\end{proof}


\begin{thebibliography}{9}
		\bibitem{AH}
		Tarik Aougab, Shinnyih Huang. \emph{Minimally intersecting filling pairs on surfaces.} Algebraic \& Geometric Topology, 15 (2015) 903-932.
		
		\bibitem{AMN}
		Tarik Aougab, William W. Menasco and Mark Nieland. \emph{Origamis associated to minimally intersecting filling pairs}, Pacific Journal of Mathematics, Vol. 317 (2022), No. 1, 1–20.

        \bibitem{AT}
		Tarik Aougab, Samuel J. Taylor. \emph{Small intersection numbers in the curve graph.} Algebraic \& Geometric Topology, 15 (2015) 903-932.

        \emph{Thurston's work on surfaces}, Princeton University Press (2021), Volume 48.

        \bibitem{C}
        Hong Chang, \emph{An upper bound of the numbers of minimally intersecting filling coherent pairs}, e-print, arXiv:2208.01126, 2022.

        \bibitem{CJM} 
        Hong Chang, Xifeng Jin, and William W. Menasco, \emph{Origami edge-paths in the curve graph}, Topology and its Applications, 2021, Vol. 298.

        \bibitem{HM}
        J. Hubbard and H. Masur. \emph{Quadratic differentials and measured foliations}. Acta Mathematica, 142:221-274, 1979.



        \bibitem{J}
		Luke Jeffreys. \emph{Minimally intersecting filling pairs on the punctured surface of genus two}, Topology and its Applications 254 (2019), 101-106.



        \bibitem{N}
		Mark Nieland. \emph{Connected-Sum Decompositions of Surfaces with Minimally-Intersecting Filling Pairs}, e-print, arXiv: 1603.03269, 2016.




  
	\end{thebibliography}
 \end{document}